\documentclass[11pt]{amsart}
\usepackage{amsmath,amssymb,amsthm, amscd}
\usepackage[all]{xy}

\newcommand\Z{\mathbb{Z}}

\newcommand\Q{\mathbb{Q}}

\newcommand\F{{\mathbb{F}}}
\newcommand{\PP}{\mathbb{P}}      % Projective space
\newcommand\Ksep{K^{\text{sep}}}
\def\p{{\mathfrak{p}}}

\DeclareMathOperator{\GL}{GL}
\DeclareMathOperator{\PGL}{PGL}

\newenvironment{Aut}{{\rm Aut}}{}
\newcommand{\invlim}{\mathop{\displaystyle \varprojlim}\limits}
\newenvironment{Gal}{{\rm Gal\,}}{}

\newtheorem{theorem}{Theorem}[section]
\newtheorem{definition}[theorem]{Definition}
\newtheorem{lemma}[theorem]{Lemma}
\newtheorem{proposition}[theorem]{Proposition}
\newtheorem{proposition-definition}[theorem]{Proposition-Definition}
\newtheorem{corollary}[theorem]{Corollary}
\newtheorem{conjecture}[theorem]{Conjecture}

\newtheorem{question}[theorem]{Question} 

\theoremstyle{definition}

\theoremstyle{remark}
\newtheorem*{remark}{Remark}

\begin{document}
\title{Eventually stable rational functions}
\author{Rafe Jones and Alon Levy}

\thanks{MSC 37P05 (primary), 37P15, 11R32, 12E05, 37P25 (secondary) \\ 
The second author's research was partially supported by the G\"{o}ran Gustafson Foundation.}

\begin{abstract}
For a field $K$, rational function $\phi \in K(z)$ of degree at least two, and $\alpha \in \PP^1(K)$, we study the polynomials in $K[z]$ whose roots are given by the solutions in $\overline{K}$ to $\phi^n(z) = \alpha$, where $\phi^n$ denotes the $n$th iterate of $\phi$. When the number of irreducible factors of these polynomials stabilizes as $n$ grows, the pair $(\phi, \alpha)$ is called eventually stable over $K$. We conjecture  that $(\phi, \alpha)$ is eventually stable over $K$ when $K$ is any global field and $\alpha$ is any point not periodic under $\phi$ (an additional non-isotriviality hypothesis is necessary in the function field case). We prove the conjecture when $K$ has a discrete valuation for which (1) $\phi$ has good reduction and (2) $\phi$ acts bijectively on all finite residue extensions. As a corollary, we prove for these maps a conjecture of Sookdeo on the finiteness of $S$-integral points in backwards orbits. We also give several characterizations of eventual stability in terms of natural finiteness conditions, and survey previous work on the phenomenon. 
\end{abstract}


\maketitle

\section{Introduction}

Given a field $K$ and $f \in K[z]$, many authors have studied the question of whether $f$ is \textit{stable over $K$}, that is, if all iterates $f^n(z)$ for $n \geq 1$ are irreducible over $K$. See for example \cite{shparostafe, danielson, ostafe2, quaddiv, itconst}, and also \cite[Sections 1 and 2]{odonigalit}, where the terminology originated. While stability is appealing in its simplicity, it cannot be expected to hold in great generality; indeed given any $f \in K[z]$ there is a finite extension $L$ of $K$ over which $f$ is not irreducible, and hence not stable. Moreover, by focusing on irreducible factors of $f^n(z)$, one is making a choice of considering only properties of the preimages of $0$ under iterates of $f$. In this paper we study a more general phenomenon in which $f$ may be a rational function, $0$ may be replaced by any element $\alpha$ of $\PP^1(K)$, and iterates of $f$ are allowed to factor non-trivially, but only finitely often.

\begin{definition}\label{eventuallystable}
Let $K$ be a field, let $\phi(z) \in K(z)$ be non-constant, and let $\alpha \in \mathbb{P}^1(K)$. For each $n \geq 1$, choose coprime $f_n, g_n \in K[z]$ with $\phi^n(z) = f_n(z)/g_n(z)$. If $\alpha \neq \infty$, we say that the pair $(\phi, \alpha)$ is \textbf{eventually stable} if the number of irreducible factors in $K[z]$ of $f_n(z) - \alpha g_n(z)$ is bounded by a constant independent of $n$. We say that $(\phi, \infty)$ is eventually stable if the number of irreducible factors of $g_n(z)$ is similarly bounded. We say $\phi$ is eventually stable if $(\phi, 0)$ is eventually stable.
\end{definition}

Note that $f_n$ and $g_n$ are determined up to a constant factor, and so the definition is independent of the choice of $f_n$ and $g_n$. We take the number of irreducible factors of a constant polynomial to be zero. Repeated irreducible factors of $f_n(z) - \alpha g_n(z)$ or $g_n(z)$ are counted according to their multiplicity. We remark that the roots of $f_n(z) - \alpha g_n(z)$ (or $g_n(z)$ if $\alpha = \infty$), counting multiplicity, are identical to the preimages (in $\overline{K}$) of $\alpha$ under $\phi^n$, again counting multiplicity. Definition \ref{eventuallystable} is a generalization of the definition in \cite[Section 4]{quaddiv}, where the terminology first appeared. 

Eventual stability is invariant under finite extension of the ground field, and has several other characterizations in terms of natural finiteness conditions, which we enumerate and prove in Section \ref{characterization}. 
%In particular, when $f_n(z) - \alpha g_n(z)$ is separable over $K$ for all $n \geq 1$, then eventual stability has an interpretation in terms of Galois orbits 
While less studied than stability, there are some results in the literature on eventual stability, notably in \cite{zdc} and \cite{ingram}. 
%which have appeared in various forms; we survey them in Section \ref{survey} [Also, say about the consequences]. 
Based in part on these results, and in part on our Theorem \ref{main}, we conjecture that eventual stability holds quite generally. In the case where $K$ is a function field with field of constants $F$, we call the pair $(\phi, \alpha)$ \textit{isotrivial} if there is $\mu \in \PGL_2(\overline{K})$ with $\mu \circ \phi \circ \mu^{-1} \in F(z)$ and $\mu(\alpha) \in \PP^1(F)$.
\begin{conjecture}[Everywhere eventual stability conjecture] \label{mainconj}
Let $\phi \in K(z)$ have degree $d \geq 2$, and suppose that $\alpha \in \PP^1(K)$ is not periodic under $\phi$.
\begin{enumerate}
\item If $K$ is a number field, then $(\phi, \alpha)$ is eventually stable over $K$.
\item If $K$ is a function field and $(\phi, \alpha)$ is not isotrivial, then $(\phi, \alpha)$ is eventually stable over $K$.
\end{enumerate} 
\end{conjecture}

Eventual stability holds trivially for linear maps, and so nothing is lost by assuming $d \geq 2$; henceforth we make this assumption for results dealing with eventual stability. Conjecture \ref{mainconj} is known just in a few cases, such as when $\phi$ is a power map and $\alpha$ is not a root of unity (see Section \ref{prior} for this and similar results for Chebyshev polynomials and some Latt\`es maps). 
When iterated preimages of $\alpha$ under $\phi$ carry a Galois action, we give in Conjecture \ref{numfieldconj} an equivalent version of part (1) of Conjecture \ref{mainconj}.

Our main result gives a new class of $\phi$ for which Conjecture \ref{mainconj} holds.   
Given a discrete non-archimedean valuation $v$ on a field $K$ (taking $v(0) = \infty$), denote by $\p$ the associated prime ideal $\{x \in K : v(x) > 0\}$ of the ring $R = \{x \in K : v(x) \geq 0\}$. Let $k$ be the residue field $R / \p$, denote by $\tilde{x}\in \mathbb{P}^1(k)$ the reduction modulo $\p$ of $x \in \PP^1(K)$  (where we take $\tilde{\infty} = \infty \in \PP^1(k)$), and denote by $\tilde{f}$ the polynomial obtained from $f \in R[z]$ by reducing each coefficient modulo $\p$. Given $\phi \in K(z)$, we can choose coprime $f, g \in R[z]$ such that $\phi(z) = f(z)/g(z)$ and at least one coefficient of $f$ or $g$ is in $R^*$. 
%(we say $\phi$ is \textit{normalized} when it is written in this form). \label{normalized} 
We take $\tilde{\phi} = \tilde{f}/\tilde{g}$, and note that the degree of $\tilde{\phi}$ is independent of the choice of $f$ and $g$. We say a non-constant $\phi \in K(z)$ has \textit{good reduction at $v$} if $\deg \tilde{\phi} = \deg \phi$, or equivalently $\tilde{f}$ and $\tilde{g}$ have no common roots.

\begin{theorem} \label{main}
Let $v$ be a discrete valuation on a field $K$, let the residue field $k$ be finite of characteristic $p$, and let $\phi \in K(z)$ have degree $d \geq 2$. Suppose that $\phi$ has good reduction at $v$ and 
\begin{equation*} 
\tilde{\phi}(z) = \frac{c_1z^{p^j} + c_2}{c_3z^{p^j} + c_4},
\end{equation*}
where $j \geq 1$ and $c_1, c_2, c_3, c_4 \in k$ (so in particular $d = p^j$). Then $(\phi, \alpha)$ is eventually stable for all $\alpha \in \PP^1(K)$ not periodic under $\phi$. %\comment{probably better to say $d \geq 2$, because it's stupid to state a theorem for $d = 1$ that requires $\alpha$ not to be periodic, when eventual stability holds for all such maps.}
\end{theorem}
In fact we give a bound on the number of irreducible factors that depends on $\phi$ and $\alpha$; see Corollary \ref{fullmain}. Maps satisfying the hypotheses of Theorem \ref{main} are precisely those that have good reduction at $v$ and the property that $\tilde{\phi}$ induces a bijection on every finite extension of $k$ (see Proposition \ref{bijective reduction characterization}).  Note that when $K$ is a function field, the finiteness of the residue field forces the field of constants to be finite and hence under the assumptions of Theorem \ref{main}, we have that $\alpha$ is periodic for $\phi$ when both are defined over the constant field. Thus $\alpha$ is periodic under $\phi$ for all isotrivial $(\phi, \alpha)$. 

Some basic cases of Conjecture \ref{mainconj} remain unresolved. We wish to draw attention to one in particular. It follows from \cite[Theorem 1.6]{zdc} (or Theorem \ref{main}) that $\phi(z) = z^2 + c \in \Q[z]$ is eventually stable over $\Q$ unless $c = 0$ or $c$ is the reciprocal of an integer. 
We conjecture that eventual stability still holds in this case, save for the obvious exceptions.
\begin{conjecture} \label{quadraticfamily}
Let $a \in \Z$ with $a \not\in \{0, -1\}$. Then $z^2 + (1/a)$ is eventually stable over $\Q$. 
\end{conjecture}

Applying Theorem \ref{main} with $K = \Q$ and $v$ the $2$-adic valuation, we prove Conjecture \ref{quadraticfamily} when $a$ is odd:
\begin{corollary}
Let $a \neq -1$ be an odd integer. Then $z^2 + (1/a)$ is eventually stable over $\Q$.
\end{corollary} 
When $a$ is even, $\phi(z) = z^2 + (1/a)$ has bad reduction for the $2$-adic valuation, and so Theorem \ref{main} does not apply. When $a = -1$, $0$ is periodic under $\phi$, and again Theorem~\ref{main} does not apply.

The exceptions in Conjecture \ref{mainconj} are necessary. First, let $\alpha$ be periodic under $\phi$, which we recall means $\phi^n(\alpha) = \alpha$ for some $n \geq 1$, and assume for simplicity that $\alpha \neq \infty$. It follows from taking $\beta = \alpha$ in equation \eqref{fundram2} in Section \ref{characterization} that $(f_{in}(z) - \alpha g_{in}(z)) \mid (f_{jn}(z) - \alpha g_{jn}(z))$ for all $i$ with $1 \leq i \leq j$. Now for $j \geq 2$, the product
\begin{equation*} \label{bigdiv}
\prod_{i = 1}^{j-1} \frac{f_{(i+1)n}(z) - \alpha g_{(i+1)n}(z)}{f_{in}(z) - \alpha g_{in}(z)} %= \frac{f_{jn}(z) - \alpha g_{jn}(z)}{f_{n}(z) - \alpha g_{n}(z)}
\end{equation*}
equals $(f_{jn}(z) - \alpha g_{jn}(z))/(f_{n}(z) - \alpha g_{n}(z))$, and hence divides $f_{jn}(z) - \alpha g_{jn}(z)$. Moreover, each term of the product is a polynomial. We claim that there are infinitely many $i$ such that $\deg (f_{(i+1)n}(z) - \alpha g_{(i+1)n}(z)) > \deg (f_{in}(z) - \alpha g_{in}(z))$, which shows that $(\phi, \alpha)$ is not eventually stable. The claim is obvious if there are infinitely many $\beta \in \overline{K}$ that map to $\alpha$ under some iterate of $f$; otherwise, the assumption that $\deg \phi \geq 2$ and elementary results on exceptional points for rational maps (see e.g. \cite[p. 807]{jhsintegers}) imply that $(\phi^{2})^{-1}(\alpha) = \{\alpha\}$. But in this case $(z - \alpha)^{d^{2i}}$ divides $f_{2i}(z) - \alpha g_{2i}(z)$ for all $i \geq 1$, from which the claim easily follows. 

Second, isotrivial functions need not be eventually stable. For example, let $\F_p$ be the finite field with $p$ elements, and take $K = \F_5(t)$ and $\phi(z) = z^2 + 2 \in K(z)$. Following \cite{ostafe}, we use a consequence of a classical theorem originally due to Pellet \cite{pellet} and rediscovered by Stickelberger \cite{stickelberger}, Voronoi, and others (see \cite[Theorem 4.11]{narkiewicz} for a modern reference): if $p$ is an odd prime, then an even-degree polynomial over $\F_p$ has an even number of irreducible factors if and only if the discriminant of the polynomial is a square in $\F_p$. Using a discriminant formula for iterates of quadratic polynomials \cite[Lemma 2.6]{quaddiv}, one sees that the discriminant of $\phi^n(z)$ is a square in $\F_5$ if and only if $\phi^n(0)$ is a square in $\F_5$. But the orbit of $0$ under $\phi$ is $0 \mapsto 2 \mapsto 1 \mapsto 3 \mapsto 1$, and thus $\phi^n(0)$ is a square in $\F_5$ precisely when $n$ is even. Thus $\phi^n(z)$ has an odd number of irreducible factors over $\F_5$ when $n$ is odd, and an even number of factors when $n$ is even; it follows that $\phi^n(z)$ has at least $n$ irreducible factors over $K$ for each $n \geq 1$, and thus $(\phi, 0)$ is not eventually stable, even though $0$ is not periodic under $\phi$. In general one does not expect an arbitrary polynomial over a finite field to be eventually stable; for a conjectural model of the factorization of iterates of such a polynomial, see \cite{settled}.

Eventual stability has several known consequences, some of which we discuss in Sections \ref{integralitysec}, \ref{preimsec}, and \ref{arbsec}. Section \ref{integralitysec} focuses on a conjecture of Sookdeo  that asserts an analogue for backwards orbits of Silverman's result on the finiteness of integer points in forwards orbits \cite{jhsintegers}. Recall that the backwards orbit $O_\phi^-(\alpha)$ under $\phi$ is by definition the union over all $n \geq 0$ of $\phi^{-n}(\alpha) := \{\beta \in \PP^1(\overline{K}) : \phi^n(\beta) = \alpha\}$. 
%, and that $\gamma \in \PP^1(\overline{K})$ is preperiodic under $\phi$ when $\phi^n(\gamma) = \phi^m(\gamma)$ for some $n > m \geq 0$. 
Let $K$ be a number field and $S$ a finite set of places of $K$ containing all archimedean places. We say that $\beta \in \PP^1(\overline{K})$ is \textit{$S$-integral with respect to $\gamma \in \PP^1(K)$} if there is no prime $\p$ of $K(\beta)$ lying over a prime outside of $S$, such that the images of $\beta$ and $\gamma$ modulo $\p$ coincide. We denote by $\mathcal{O}_{S, \gamma}$ the set of all $\beta \in \PP^1(\overline{K})$ that are $S$-integral with respect to $\gamma$. As an example, if $S$ consists only of the archimedean places of $K$, then $\mathcal{O}_{S, \infty}$ is the ring of algebraic integers in $\overline{K}$. Sookdeo conjectures \cite[Conjecture 1.2]{sookdeo} that $\mathcal{O}_{S, \gamma} \cap O_\phi^{-}(\alpha)$ is finite unless $\gamma$ is preperiodic for $\phi$. Theorem \ref{main} and a result due to Sookdeo (see Theorem \ref{backorb}) allow us to prove:
\begin{corollary} \label{integralitycor}
Let $K$ be a number field, $S$ a finite set of places of $K$ containing all archimedean places, and $\alpha \in \PP^1(K)$. If $\phi \in K(z)$ satisfies the hypotheses of Theorem \ref{main}, then % and $\alpha$ is not periodic under $\phi$, then
%\begin{equation*}
$\mathcal{O}_{S, \gamma} \cap O_\phi^{-}(\alpha)$ is finite for all $\gamma \in \PP^1(K)$ not preperiodic under $\phi$. 
%\end{equation*}
\end{corollary}
The exclusion of preperiodic $\gamma$ is necessary; for instance, in the case where $\phi$ is a monic polynomial, $\gamma = \infty$, and $S$ consists of the archimedean places of $K$, we have $O_\phi^{-}(\alpha) \subset \mathcal{O}_{S, \gamma}$. Note that Corollary \ref{integralitycor} holds even when $\alpha$ is periodic under $\phi$. See Section \ref{integralitysec} for a proof.

The heart of our method is to generalize one of the most fundamental facts about stability: Eisentein polynomials are stable (see e.g. \cite[Lemma 2.2 (iii)]{odonigalit} for a result in a very general setting). Here, we show that eventually stability holds for rational functions satisfying a weak version of the Eisenstein criterion:
\begin{theorem} \label{evstab2intro}
Let $v$ be a discrete valuation on $K$, let $\phi \in K(z)$ have degree $d \geq 2$, and let $\alpha \in \mathbb{P}^1(K)$. Suppose that $\phi$ has good reduction at $v$, $\phi(\alpha) \neq \alpha$, and $\tilde{\phi}^{-1}(\tilde{\alpha}) = \{\tilde{\alpha}\}$ as a map of $\mathbb{P}^1(\overline{k})$. Then $(\phi, \alpha)$ is eventually stable over $K$. 
\end{theorem}
%The condition that $f(z) -\alpha g(z) \equiv g(z)(z - \alpha)^d \bmod{\p}$ is equivalent to requiring that the coefficient-wise reduction $\tilde{\phi} \in (R / \p)[z]$ have degree $d$ and $\tilde{\alpha} \in \mathbb{P}^1(R / \p)$ as a totally ramified fixed point. 
We give a more precise result in Theorem \ref{evstab2}, where we give bounds in terms of $\phi$ and $\alpha$ on the number of irreducible factors of the relevant polynomials.
In the case where $\alpha = 0$ and $\phi(z) = f(z) \in R[z]$ has degree $d$, Theorem \ref{evstab2intro} says that $f(z)$ is eventually stable provided that $\tilde{f}(z) = cz^d$ for $c \in k \setminus \{0\}$, or in other words when $\p$ divides all coefficients of $f$ except its leading coefficient; this is the aforementioned weak version of the Eisenstein criterion.

\section{Characterizations of eventual stability} \label{characterization}

The recent literature includes several articles where eventual stability appears in different guises (e.g.  \cite{zdc}, \cite{ingram}, \cite{sookdeo}). In this section we show these guises are all equivalent (Propositions \ref{charprop1}, \ref{charprop2}), give in Conjecture \ref{numfieldconj} a reformulation of part (1) of Conjecture \ref{mainconj} in the setting where iterated preimages of $\alpha$ carry a Galois action, and discuss the related notion of settledness studied in \cite{settled} (see Question \ref{finfldquestion}).

We begin with some general remarks about eventual stability. Let $K$ be a field, let $\phi(z) \in K(z)$ have degree $d \geq 1$, let $\alpha \in \PP^1(K)$, and for each $n \geq 1$ let $f_n, g_n \in K[z]$ be coprime polynomials with $\phi^n(z) = f_n(z)/g_n(z)$. As noted in the introduction, the roots of $f_n(z) - \alpha g_n(z)$, counting multiplicity, are identical to the preimages (in $\overline{K}$) of $\alpha$ under $\phi^n$, again counting multiplicity: \begin{equation} \label{fundram1}
f_n(z) - \alpha g_n(z) = C \prod_{\beta \in \overline{K} : \phi^n(\beta) = \alpha} (z - \beta)^{e_{n}(\beta)},
\end{equation}
where $C \in K \setminus \{0\}$ and we write $e_n(\beta)$ for $e_{\phi^n}(\beta)$, the ramification index of $\phi^n$ at $\beta$ (i.e. the order of vanishing at $z = \beta$ of $\phi^n(z) - \phi^n(\beta)$, with suitable modifications for $\beta = \infty$; see \cite[Section 1.2]{jhsdynam}). For $\alpha = \infty$, we have a similar statement, but with $g_n(z)$ replacing $f_n(z) - \alpha g_n(z)$.  
Given rational functions $\phi, \psi \in K(z)$ and $\gamma \in \overline{K}$, an easy argument on compositions of power series (see \cite[Section 2.5]{Beardon}) gives 
$e_{\phi \circ \psi}(\gamma) = e_\phi(\psi(\gamma)) \cdot e_\psi(\gamma)$, and hence for all $m \geq 1$ we have $e_{\phi^{n+m}}(\gamma) = e_{\phi^n}(\phi^m(\gamma)) \cdot e_{\phi^m}(\gamma)$. It follows that, up to a non-zero multiplicative constant, we have 

\begin{align} 
f_{n+m}(z) - \alpha g_{n+m}(z) & =  \prod_{\gamma \in \overline{K} : \phi^{n+m}(\gamma) = \alpha} (x - \gamma)^{e_{n+m}(\gamma)} \nonumber \\
& =  \prod_{\beta \in \PP^1(\overline{K}) : \phi^{n}(\beta) = \alpha} \left( \prod_{\gamma \in \overline{K} : \phi^m(\gamma) = \beta} (x - \gamma)^{e_{n}(\beta)e_m(\gamma)} \right) \nonumber \\
& = g_m(z)^{\epsilon_n} \prod_{\beta \in \overline{K} : \phi^n(\beta) = \alpha} (f_m(z) - \beta g_m(z))^{e_n(\beta)}, \label{fundram2}
\end{align}
where $\epsilon_n = e_{\phi^n}(\infty)$ if $\phi^n(\infty) = \alpha$ and $\epsilon_n = 0$ otherwise. Factoring out $g_m(z)^{e_n(\beta)}$ from each term of the product in \eqref{fundram2}, and using \eqref{fundram1}, we obtain
\begin{equation} \label{fundram3}
f_{n+m}(z) - \alpha g_{n+m}(z) = C g_m(z)^{d^n}[ f_n(f_m(z)/g_m(z)) - \alpha g_n(f_m(z)/g_m(z))]
\end{equation}
for $C \in K \setminus \{0\}$. When $\alpha = \infty$ we obtain the similar equation $$g_{n+m}(z) = Cg_m(z)^{d^n} [g_n(f_m(z)/g_m(z))].$$ 
These furnish a mild generalization of the equations derived in \cite[Section 2]{galrat}.

Having dispensed with these preliminaries, we now state our characterization:

\begin{proposition} \label{charprop1}
Let $K$ be a field, let $\phi(z) \in K(z)$ have degree $d \geq 2$, and let $\alpha \in \PP^1(K)$. For each $n \geq 1$, choose coprime $f_n, g_n \in K[z]$ with $\phi^n(x) = f_n(z)/g_n(z)$, and let $(\beta_n)_{n \geq 1}$ be a sequence of elements of $\PP^1(\overline{K})$ satisfying $\phi(\beta_1) = \alpha$ and $\phi(\beta_n) = \beta_{n-1}$ for $n \geq 2$. The following are equivalent:
\begin{enumerate}
\item[$(a)$] $(\phi, \alpha)$ is eventually stable over $K$.
\item[$(b)$]  $(\phi, \alpha)$ is eventually stable over $L$, for any finite extension $L$ of $K$.  
\item[$(c)$] There exists $n \geq 1$ such that $(\phi^n, \alpha)$ is eventually stable over $K$.
\item[$(d)$]  For all $\mu \in \PGL(2, \overline{K})$, $(\mu \circ \phi \circ \mu^{-1}, \mu(\alpha))$ is eventually stable over the minimal extension of $K$ containing the coefficients of $\mu$. 
\item[$(e)$]  $[K(\beta_{n+1}) : K(\beta_{n})] = d$ for all $n \geq M$, where $M$ depends only on $\phi$ and $\alpha$. 
\item[$(f)$]   $[K(\beta_n) : K] \geq Cd^n$ for some $C > 0$ depending only on $\phi$ and $\alpha$. 
\end{enumerate}
\end{proposition}

Note that in $(e)$ and $(f)$ of Proposition \ref{charprop1}, we regard $K(\infty)$ as identical to $K$. Before proving Proposition \ref{charprop1}, we give three more equivalent conditions for eventual stability. In many settings the roots of $f_n(z) - \alpha g_n(z)$ lie in the separable closure $\Ksep$ of $K$, for each $n \geq 1$ (note that this does not require the roots of $f_n(z) - \alpha g_n(z)$ to be distinct). In this case we say that the pair $(\phi, \alpha)$ is separable. \label{separable}
%Note that this seems to hold independent of $\alpha$, i.e. it's a property of $\phi$ alone. Could develop this further.
When this holds, the absolute Galois group $G_K := \Gal(\Ksep / K)$ acts naturally on each set $\phi^{-n}(\alpha) := \{\beta \in \PP^1(\overline{K}) : \phi^n(\beta) = \alpha\}$, where we regard $\infty$ as defined over $K$ and hence fixed by $G_K$. Indeed, this action extends to an action by tree automorphisms on \label{tree}
$$T := \bigsqcup_{n \geq 1} \phi^{-n}(\alpha),$$
which becomes a rooted tree if we assign edges according to the action of the maps $\phi^{-n}(\alpha) \to \phi^{-n+1}(\alpha)$ induced by $\phi$. Moreover, $T$ is an inverse system under these maps, and we put $\delta T := \invlim \phi^{-n}(\alpha),$ which is sometimes known as the set of \textit{ends} of $T$. An element of $\delta T$ is a sequence $(\beta_n)_{n \geq 1}$ with $\phi(\beta_1) = \alpha$ and $\phi(\beta_n) = \beta_{n-1}$ for each $n \geq 2$. Giving $\phi^{-n}(\alpha)$ the discrete topology, we have that $\delta T$ is a compact topological space. 
%Proof: the spaces $\phi^{-n}(\alpha)$ with the discrete topology are Hausdorff, and so the inverse limit is a closed subset of the infinite product of the $\phi^{-n}(\alpha)$. This product is compact by Tychonoff's theorem, and a closed subspace of a compact space is compact. 
A basis for this topology is given by sets of the form $\pi_n^{-1}(S)$ for $S \subset \phi^{-n}(\alpha)$, where 
$\pi_n : \delta T \to \phi^{-n}(\alpha)$ is the natural projection. Note that $\pi_n^{-1}(S)$ is also closed, since its complement is the open set $\pi_n^{-1}(S^c)$. When $(\phi, \alpha)$ is separable, the action of $G_K$ on $T$ extends to an action on $\delta T$, with $\sigma((\beta_n)_{n \geq 1}) = (\gamma_n)_{n \geq 1}$ if and only if $\sigma(\beta_n) = \gamma_n$ for all $n \geq 1$. 

\begin{proposition} \label{charprop2}
With assumptions as in Proposition \ref{charprop1}, suppose additionally that $(\phi, \alpha)$ is separable. Then the following are equivalent:
\begin{enumerate}
\item[(A)] $(\phi, \alpha)$ is eventually stable over $K$. 
\item[(B)] The number of $G_K$-orbits on $\phi^{-n}(\alpha)$ is bounded as $n$ grows. 
\item[(C)] The number of $G_K$-orbits on $\delta T$ is finite.
\item[(D)] Every $G_K$-orbit on $\delta T$ is open.
\end{enumerate}
\end{proposition}

We remark that when $K$ is a global field and $(\phi, \alpha)$ is separable, there is a finite $G_K$-orbit on $\delta T$ if and only if $\alpha$ is periodic under $\phi$, and such an orbit must consist of a single point (the element of $\delta T$ corresponding to the cycle containing $\alpha$). Indeed, let $(\beta_n)_{n \geq 1} \in \delta T$ have a finite orbit under $G_K$, and note that the stabilizer of $(\beta_n)$ in $G_K$ must have finite index in $G_K$, and thus $L = K(\beta_1, \beta_2, \ldots)$ is a finite extension of $K$. Moreover, if $\hat{h}_\phi$ is the canonical height associated to $\phi$ (see \cite[Section 3.4]{jhsdynam} for the definition and basic properties over number fields, and \cite{langdioph} for a more general treatment), then because $\hat{h}_\phi(\phi(\gamma)) = d \hat{h}_\phi(\gamma)$ for all $\gamma \in \overline{K}$, it follows that $\{\beta_1, \beta_2, \ldots\}$ is a set of bounded height, and hence has finite intersection with $L$. Therefore $\{\beta_1, \beta_2, \ldots\}$ is finite,
implying that $\beta_i = \beta_j$ for some $i \neq j$, and thus $\alpha = \beta_n$ for some $n$, implying the desired statements. We remark that similar conclusions hold when $K$ is a finite field.

These remarks and part (D) of Proposition \ref{charprop2} imply that the following conjecture is equivalent to part (1) of Conjecture \ref{mainconj}:
\begin{conjecture} \label{numfieldconj}
If $K$ is a number field, then every $G_K$-orbit on $\delta T$ is either open or finite. %If $K$ is a function field, the same conclusion holds unless $(\phi, \alpha)$ is non-separable or isotrivial. 
\end{conjecture}

As mentioned in the introduction, one does not expect eventual stability to hold in general when $K$ is a finite field, as discussed in \cite{settled}. Rather, in \cite{settled} the authors propose the weaker condition of \textit{settledness}, which we can now describe in succinct fashion. %We can now give a conveniently stated generalization of \cite[Conjecture 2.2]{settled} to rational functions of arbitrary degree. 
Let $\mu_n$ be the probability measure on $\phi^{-n}(\alpha)$ that assigns equal mass to each point, counting multiplicity (so a root of multiplicity $m$ gets $m$ times the mass of a root of multiplicity 1). We then obtain a probability measure $\mu$ on $\delta T$ by assigning $\mu(\Sigma) = \lim_{n \to \infty} \mu_n(\pi_n(\Sigma))$. That $\mu$ is a measure is not trivial, but follows from pulling back each $\mu_n$ to $\delta T$ in the obvious way and invoking the Vitali-Hahn-Saks theorem (see e.g. \cite[p. 277]{jacobs}). We say that the pair $(\phi, \alpha)$ is \textit{settled} if the union $U$ of the open $G_K$-orbits on $\delta T$ satisfies $\mu(U) = 1$; compare to \cite[Definition 2.1]{settled}. 
\begin{question} \label{finfldquestion}
Let $K$ be a finite field, $\phi \in K(x)$ of degree $d \geq 2$, and assume that the characteristic of $K$ does not divide $d$. Must $(\phi, \alpha)$ be settled for all $\alpha \in \PP^1(K)$?
\end{question}
Very little is known about Question \ref{finfldquestion}. In the case where $d = 2$ and $\phi$ is a polynomial, Conjecture 2.2 of \cite{settled} asserts a positive answer to Question \ref{finfldquestion}, and gives some evidence.

\begin{proof}[Proofs of Propositions \ref{charprop1} and \ref{charprop2}]
We first show $(a) \Rightarrow (e)$. We suppose that $\alpha \neq \infty$; otherwise a similar argument holds with $f_i(z) - \alpha g_i(z)$ replaced by $g_i(z)$. By $(a)$ we may take $M$ large enough so that $f_n(z) - \alpha g_n(z)$ has the same number of irreducible factors over $K$, for all $n \geq M$. Furthermore, there can be at most one $n$ with $\phi^n(\infty) = \alpha$, for otherwise $\alpha$ is periodic under $\phi$, and $(a)$ is violated. Hence, increasing $M$ if necessary, we may assume that $\infty \not\in \phi^{-n}(\alpha)$ (and in particular $\beta_n \neq \infty$) for all $n \geq M$. Now fix $n \geq M$, let $h_1, \ldots, h_r \in K[z]$ be the irreducible factors of $f_n(z) - \alpha g_n(z)$, and let $d_1, \ldots, d_r$ be their degrees. 
From \eqref{fundram3} we have 
\begin{equation} \label{breakdown}
f_{n+1}(z) - \alpha g_{n+1}(z) = C g(z)^{d^n}[ f_n(\phi(z)) - \alpha g_n(\phi(z))] = C \prod_{j = 1}^r  g(z)^{d_j} h_j(\phi(z)),
\end{equation}
and because $\phi^{n+1}(\infty) \neq \alpha$, we have that $s_j(z) := g(z)^{d_j} h_j(\phi(z))$ is a polynomial of degree $d_jd$. Note that \eqref{breakdown} already gives a factorization of $f_{n+1}(z) - \alpha g_{n+1}(z)$ into the same number of irreducible factors over $K$ as $f_n(z) - \alpha g_n(z)$, and thus each $s_j(z)$ is irreducible over $K$. Because $\beta_{n+1} \neq \infty$, it follows from \eqref{fundram1} that there is a unique $s_\ell$ with $\beta_{n+1}$ as a root, and thus $[K(\beta_{n+1}) : K] = d_\ell d$. But then $\beta_n$ is a root of $h_\ell$, and so we must have $[K(\beta_n) : K] = d_\ell$ and $[K(\beta_{n+1}) : K(\beta_n)] = d$, as desired.  

To show $(e)$ $\Rightarrow$ $(f)$, it suffices to prove that there exists $C_0 > 0$ depending only on $\phi$ and $\alpha$ such that $[K(\beta_n): K] = C_0d^n$ for all $n$ sufficiently large. Let $M$ be as in the previous paragraph. For any $n \geq M$, condition $(e)$ and multiplicativity of degrees give $[K(\beta_n) : K] = C_0d^n$ with $C_0 = d^{-M}[K(\beta_M) : K]$, proving $(f)$.

We now prove $(f) \Rightarrow (a)$. Suppose that there are infinitely many $n$ with $[K(\beta_{n+1}) : K(\beta_n)] < d$. Then given $C > 0$ we may take $i$ with $((d-1)/d)^i < C$, and we may also take $m$ large enough so that $[K(\beta_{n+1}) : K(\beta_n)] < d$ for at least $i$ values of $n$ that are less than $m$. Then $[K(\beta_m) : K] \leq (d-1)^id^{m-i} < Cd^m$, contradicting $(f)$. Thus there exists $M$ such that $[K(\beta_{n+1}) : K(\beta_n)] = d$ for all $n \geq M$, and because $C$ is independent of $(\beta_n)$ in $(f)$, $M$ is also independent of $(\beta_n)$. Further, increase $M$ if necessary so that $\phi^n(\infty) \neq \alpha$ for all $n \geq M$; this is possible because if $\phi^n(\infty) = \alpha$ for more than one value of $n$, then $\alpha$ is periodic under $\phi,$ and thus there is a choice of $(\beta_n)$ with $K(\beta_n) = K$ for all $n$. We now argue that for $n \geq M$, $f_{n+1}(z) - \alpha g_{n+1}(z)$ has the same number of irreducible factors as $f_{n}(z) - \alpha g_{n}(z)$, which proves $(a)$. If not, then one of the polynomials on the right-hand side of \eqref{breakdown} is reducible over $K$, and letting $\beta_{n+1}$ be a root of such a polynomial and $\beta_n$ a root of the corresponding $h_j$, we have $[K(\beta_{n+1}) : K(\beta_n)] < d$, a contradiction. 

The implications $(b) \Rightarrow (a)$, $(a) \Rightarrow (c)$, and $(d) \Rightarrow (a)$  are obvious. To show $(a) \Rightarrow (b)$, suppose that $h(z) \in K[z]$ is irreducible over $K$ and has $m$ irreducible factors over $L$, with $e$ being the smallest degree of these factors. Then there exists a degree $e$ extension $L'$ of $L$ containing a root of $h$, whence $[L' : K] \geq d$. This implies $[L : K] \geq d/e \geq m$, and so $h$ has at most $[L : K]$ irreducible factors over $L$. It follows that if $h$ has $r$ irreducible factors over $K$, then it has at most $r[L : K]$ irreducible factors over $L$. To show $(c) \Rightarrow (a)$, let $(\beta_i)$ be a sequence with $\phi(\beta_{1}) = \alpha$ and $\phi(\beta_i) = \beta_{i-1}$ for $i \geq 2$, %so that $(\beta_{in})_{i \geq 1}$ satisfies $\phi^n(\beta_{}
Assuming (c), we note that $\deg \phi^n = d^n$, and we may invoke condition $(e)$ to obtain that for $j \geq M$, $[K(\beta_{n(j+1)}) : K(\beta_{nj})] = d^n$. But this implies that $[K(\beta_{i+1}) : K(\beta_i)] = d$ for all $i \geq nM$, which is equivalent to $(a)$.
To show $(a) \Rightarrow (d)$, let $L$ be the minimal extension of $K$ containing the coefficients of $\mu$, which is a finite extension of $K$. By $(b)$ we have that $(\phi, \alpha)$ is eventually stable over $L$. 
%note that a given $\mu \in \PGL_2(\overline{K})$ is defined over a finite extension of $K$, but replacing $K$ with a finite extension preserves eventual stability by $(b)$, and so without loss we may take $\mu \in \PGL_2(K)$. 
Let $\phi^{\mu} = \mu \circ \phi \circ \mu^{-1}$, and let $(\gamma_n)$ be a sequence with $\phi^{\mu}(\gamma_{1}) = \mu(\alpha)$ and $\phi^{\mu}(\gamma_{n}) = \gamma_{n-1}$ for $n \geq 2$. Note that $\phi(\mu^{-1}(\gamma_1)) = \alpha$ and $\phi(\mu^{-1}(\gamma_{n})) = \mu^{-1}(\gamma_{n-1})$ for $n \geq 2$. The equivalence of $(a)$ and $(e)$ implies that $[L(\mu^{-1}(\gamma_{n})) : L(\mu^{-1}(\gamma_{n-1}))] = d$ for $n$ large enough. Observe now that $L(\mu^{-1}(\gamma_{i})) = L(\gamma_i)$ for all $i$, and invoking the equivalence of $(a)$ and $(e)$ again we have that $(\phi^\mu, \mu(\alpha))$ is eventually stable over $L$, as desired. 

We turn now to Proposition \ref{charprop2}. Note that by \eqref{fundram1} the number of $G_K$-orbits on $\phi^{-n}(\alpha)$ is the same as the number of irreducible factors of $f_n(z) - \alpha g_n(z)$, with two caveats: we must ignore multiplicity when counting irreducible factors of $f_n(z) - \alpha g_n(z)$, and if $\infty \in \phi^{-n}(\alpha)$, then there is a single $G_K$-orbit on $\phi^{-n}(\alpha)$ that does not correspond to an irreducible factor of $f_n(z) - \alpha g_n(z)$. However, by Riemann-Hurwitz only finitely many $e_n(\beta)$ are greater than one, and hence the number of $G_K$-orbits on $\phi^{-n}(\alpha)$ differs from the number of irreducible factors of $f_n(z) - \alpha g_n(z)$ (counted with multiplicity) by a number bounded independent of $n$.  Hence conditions (A) and (B) are equivalent. 

Note that (D) implies (C) by the compactness of $\delta T$ and the fact that the $G_K$-orbits on $\delta T$ furnish an open cover that is also a partition of $\delta T$, and hence has no proper subcovers. 

Observe that the projection $\pi_n$ commutes with the action of $G_K$, and thus maps $G_K$-orbits on $\delta T$ to $G_K$-orbits on $\phi^{-n}(\alpha)$. Because $\pi_n$ is a surjection, we cannot have more $G_K$-orbits on $\phi^{-n}(\alpha)$ than on $\delta T$, whence (C) implies (B).

Finally, to show (B) implies (D), take $M$ so that the number of $G_K$-orbits on $\phi^{-n}(\alpha)$ is the same for all $n \geq M$. We claim that the inverse image under $\pi_M$ of a $G_K$-orbit on $\phi^{-M}(\alpha)$ is a $G_K$-orbit on $\delta T$, which is enough to deduce (D). Let $(\beta_n)_{n \geq 1}$ and $(\gamma_n)_{n \geq 1}$ be two elements of $\pi_M^{-1}(O)$, where $O$ is a $G_K$-orbit on $\phi^{-M}(\alpha)$. In order for the number of $G_K$-orbits on $\phi^{-n}(\alpha)$ not to grow for $n \geq M$, we must have that elements of $\phi^{-n}(\alpha)$ that restrict to $O$ form a $G_K$-orbit. Hence for each $n \geq M$ there is $\sigma_n \in G_K$ with $\sigma_n(\beta_n) = \gamma_n$ and $\sigma_n|_L = \text{id}$ if $L/K$ is Galois and $\beta_n \not\in L$, and indeed the same conclusion holds for $n \geq 1$. Because $G_K$ is the inverse limit over $\Gal(L/K)$ for finite Galois extensions $L$ of $K$, there exists $\sigma \in G_K$ with $\sigma(\beta_n) = \gamma_n$ for all $n$. Hence $\pi_M^{-1}(O)$ is a $G_K$-orbit on $\delta T$. 
\end{proof}

\section{Applications of and results on eventual stability: a brief survey} \label{survey}

\subsection{Applications to relative $S$-integrality of iterated preimages} \label{integralitysec} Let $\mathcal{O}_{S, \gamma}$ and $O_\phi^{-}(\alpha)$ be defined as on p. \pageref{integralitycor}, and recall that $\gamma$ is preperiodic under $\phi$ when $\phi^n(\gamma) = \phi^m(\gamma)$ for some $n > m \geq 0$. 
The following is a result of Sookdeo \cite[Theorems 2.5 and 2.6]{sookdeo}, though there the result is stated using condition (B) of Proposition \ref{charprop2}.
\begin{theorem}[\cite{sookdeo}] \label{backorb}
Let $K$ be a number field, $S$ a finite set of places of $K$ containing all archimedean places, $\alpha \in \PP^1(K)$, and $\phi \in K(z)$ have degree $d \geq 2$.  If $(\phi, \alpha)$ is eventually stable, then
\begin{equation} \label{finint}
\text{$\mathcal{O}_{S, \gamma} \cap O_\phi^{-}(\alpha)$ is finite for all $\gamma \in \PP^1(K)$ not preperiodic under $\phi$.} 
\end{equation}
\end{theorem}
Theorem \ref{backorb} is a natural analogue for backwards orbits of Silverman's result on the finiteness of integer points in forwards orbits \cite{jhsintegers}. Sookdeo conjectures \cite[Conjecture 1.2]{sookdeo} that \eqref{finint} holds for all $\alpha \in \PP^1(K)$. Thus part (1) of our Conjecture \ref{mainconj} implies Sookdeo's conjecture in the case where $\alpha$ is not periodic under $\phi$. We remark that one method Sookdeo gives of proving Theorem \ref{backorb} in the case where $\alpha$ is not preperiodic (other methods are required when $\alpha$ is preperiodic but not periodic) is to derive \eqref{finint} under the assumption that 
\begin{equation} \label{heightbound}
\hat{h}_\phi(\beta) \geq \frac{\epsilon}{[K(\beta) : K]},
\end{equation}
for all $\beta \in O_\phi^-(\alpha)$, where $\epsilon > 0$ depends on $\phi$, $K$, and $\alpha$, but not $\beta$, and $\hat{h}_\phi$ is the canonical height associated to $\phi$. From the fact that $\hat{h}_{\phi}(\phi^n(\beta)) = d^n\hat{h}_\phi(\beta)$, it follows that \eqref{heightbound} is equivalent to condition $(f)$ of Proposition \ref{charprop1}. The Dynamical Lehmer Conjecture \cite[Conjecture 3.25]{jhsdynam} asserts that \eqref{heightbound} holds for $\epsilon$ depending only on $\phi$ and $K$, and hence implies our Conjecture \ref{mainconj}. The bound given in Theorem \ref{fullmain} on the number of irreducible factors allows one to obtain information about the constant $\epsilon$ in \eqref{heightbound} for certain $\phi$, but the bound depends on $\alpha$. 

\begin{proof}[Proof of Corollary \ref{integralitycor}] If $\alpha$ is not periodic under $\phi$, then $(\phi, \alpha)$ is eventually stable over $K$ by Theorem \ref{main}, and the Corollary follows immediately from Theorem \ref{backorb}. Suppose that $\alpha$ is periodic under $\phi$, note that $O_\phi(\alpha)$ consists of the cycle containing $\alpha$, and observe that the set $B := \phi^{-1}(O_\phi(\alpha)) \setminus O_\phi(\alpha)$ is finite and consists of points not periodic under $\phi$. But
$$
O_\phi^-(\alpha) = O_\phi(\alpha) \cup \bigcup_{\beta \in B} O_\phi^-(\beta),
$$
and by Theorems \ref{main} and \ref{backorb}, $\mathcal{O}_{S, \gamma} \cap O_\phi^-(\beta)$ is finite for each $\beta \in B$. 
\end{proof}

\subsection{Applications to preimage curves}  \label{preimsec}

Let $K$ be a field, $C$ a curve over $K$, and $K(C)$ the function field of $C$. Suppose that $\phi \in K(C)(z)$ and $\alpha \in K(C)$. Given $t \in C$, we denote the specializations of $\phi$ and $\alpha$ above $t$, should they be defined, by $\phi_t(z) \in K(z)$ and $\alpha_t \in \PP^1(K)$. Define the \textit{$N$th preimage curve} $X^{\text{Pre}}_{\phi, \alpha}(N)$ to be a smooth projective model of $\{\phi_t^N(z) = \alpha_t\} \subset \PP^1 \times C$. In \cite{preimcurves}, the authors studied the geometry of these curves in the special case where $\text{char}(K) \neq 2$, $C = \mathbb{A}^1$, $\phi(z) = z^2 + t$, and $\alpha \in K$ is a constant function. This geometric analysis led to results on the finiteness of the number of rational iterated preimages of $a \in \Q$ under $z^2 + t$ as $t \in \Q$ varies; indeed a uniform bound was obtained that holds for most $a \in \Q$. As noted in \cite[p. 303]{ingram}, a first step towards generalizing this result to other choices for $\phi$ and $\alpha$ is to show that $X^{\text{Pre}}_{\phi, \alpha}(N)$ is well-behaved geometrically. At the very least, one would hope that the number of irreducible components of $X^{\text{Pre}}_{\phi, \alpha}(N)$ is bounded as $N$ grows. 

\begin{proposition} \label{boundedcomps}
Let $K$ be the function field of a curve over an algebraically closed field, let $\phi \in K(z)$ have degree $d \geq 2$, and let $\alpha \in K$. If $(\phi, \alpha)$ is separable (see p. \pageref{separable}) and eventually stable, then the number of irreducible components of the $N$th preimage curve $X^{\text{Pre}}_{\phi, \alpha}(N)$ is eventually constant as $N \to \infty$. 
\end{proposition}
Because the result is geometric, we have taken the constant field to be algebraically closed. Proposition \ref{boundedcomps} is an immediate consequence of condition (B) of Proposition \ref{charprop2}.

\subsection{Applications to arboreal Galois representations} \label{arbsec}

When $(\phi, \alpha)$ is separable, Proposition \ref{charprop2} shows that eventual stability gives some coarse information about the size of the image of the homomorphism $\omega : G_K \to \Aut(T)$ (notation as on p. \pageref{tree}), also known as the \textit{arboreal Galois representation} attached to $(\phi, \alpha)$. That is, the image of $\omega$ is  not ``too small" in the sense that as $n$ grows it acts with a bounded number of orbits on $\phi^{-n}(\alpha)$. Thus eventual stability is a stepping stone to the deeper problem of showing that the image of $\omega$ has finite index in $\Aut(T)$. An easy argument shows that such finite index results imply eventual stability (see Proposition \ref{elem}), and thus it is particularly surprising that in some cases eventual stability is enough to imply finite-index results. For instance, assuming the abc conjecture, this is true for certain quadratic polynomials over $\Q$ \cite[Section 6]{abc}, and work of Bridy and Tucker \cite{bridytucker} shows it holds for large classes of cubic polynomials as well, under the additional assumption of Vojta's conjecture for surfaces. Even without assuming any conjectures, one can sometimes use results such as Siegel's theorem to deduce from eventual stability significant information about the image of $\omega$, often enough to obtain zero-density results for prime divisors of orbits (see \cite[discussion preceding Theorem 4.3]{survey}). 

%An easy argument shows that such finite index results imply eventual stability (see Proposition \ref{elem}). It is thus perhaps surprising that in some cases \textit{stability} is enough to imply finite-index results. For instance, assuming the abc conjecture, this is true for certain quadratic polynomials over $\Q$ \cite[Section 6]{abc}, and preliminary observations of Bridy and Tucker suggest it holds for large classes of cubic polynomials as well. 

%An important application of arboreal Galois representations is in proving zero-density results for prime divisors of orbits (see e.g. \cite[Sections 1 and 4]{survey}. While it is not currently known (even assuming the abc conjecture) whether eventual stability implies that $\omega(G_K)$ has finite index in $\Aut(T)$, one can \textit{unconditionally} deduce from eventual stability significant information about $\omega(G_K)$, often enough to obtain the aforementioned zero-density results. For more details, see \cite[discussion preceding Theorem 4.3]{survey} and \cite[Section 4]{quaddiv}. 

\subsection{Prior results on eventual stability} \label{prior}

The two most general results on eventual stability prior to the present paper are the following. First, we have \cite[Corollary 3]{ingram}: let $K$ be a number field, $\phi(x) \in K[x]$ monic of degree $d \geq 2$, and suppose there is a prime ${\mathfrak{p}}$ of $K$ with ${\mathfrak{p}} \nmid d$ and $v_{\mathfrak{p}}(\phi^n(\alpha)) \to 
-\infty$ as $n \to \infty$. Then $(\phi, \alpha)$ is eventually stable over $K$. Second, \cite[Theorem 1.6]{zdc}: let $d \geq 2$, let $K$ be a field of characteristic not dividing $d$, and let $\phi(x) = x^d + c \in K[x]$. If there is a discrete valuation $v$ on $K$ with $v(c) > 0$, then $(\phi, 0)$ is eventually stable over $K$. The second result is an immediate consequence of our Theorem \ref{evstab2}. The first result, which is a corollary of the stronger Theorem 1 of \cite{ingram} giving information about the $\p$-adic Galois representation attached to $(\phi, \alpha)$, also follows from Theorem \ref{evstab2} (see Corollary \ref{polycor}).

Additional eventual stability results can be found for certain one-parameter families of quadratic polynomials. See for instance Propositions 4.5, 4.6, and 4.7 of \cite{quaddiv}, and note that similar results can likely be proved for many other families using \cite[Proposition 4.2]{quaddiv}.

To the authors' knowledge, there is only one other setting where eventual stability results are known: those rare cases when detailed results are available on the image of the arboreal representation $\omega : G_K \to \Aut(T)$. These results often show that the image of the homomorphism $G_K \to \Aut(T)$ has finite index in some prescribed subgroup $G \leq \Aut(T)$ that acts transitively on $\delta T$. The following elementary proposition shows this implies eventual stability, using condition (C) of Proposition \ref{charprop2}:

\begin{proposition} \label{elem}
Let $G$ be a group acting transitively on a set $S$, and let $H$ be a subgroup of $G$ whose action on $S$ has at least $t$ orbits. Then $[G:H] \geq t$. 
\end{proposition}

\begin{proof} 
Select elements $s_1, \ldots s_t \in S$ from distinct orbits of $H$. By transitivity of the action of $G$ on $S$, there exist $g_1, \ldots, g_t \in G$ with $g_i(s_1) = s_i$ for all $i$. Let $i \neq j$, and note that $g_ig_j^{-1}$ maps $s_j$ to $s_i$. Because $s_i$ and $s_j$ are in different orbits of $H$, we must have $g_ig_j^{-1} \not\in H$, proving that $g_1H, \ldots, g_tH$ are distinct.
\end{proof} 

Finite-index results of the kind mentioned above are known principally in the case where $\phi$ is \textit{dynamically affine}, in the terminology of \cite[Section 6.8]{jhsdynam}. That is, there is a semiabelian variety $A$ and $\delta : A \to A$ obtained by composing an endomorphism and a translation on $A$, and a finite separable morphism $\pi : A \to \PP^1$ such that the following diagram commutes:
\begin{equation*}
  \xymatrix@R+2em@C+2em{
  A \ar[r]^{\delta} \ar[d]_{\pi} & A \ar[d]^{\pi} \\
  \PP^1 \ar[r]_{\phi} & \PP^1
  }
\end{equation*}
We examine three well-known kinds of rational functions that arise by taking $\delta$ to be multiplication by an integer. %These are some of the best-known examples. 

\smallskip

\noindent \textbf{Case 1}: $A = \mathbb{G}_m$, $\delta(z) = z^d$ for $d \geq 2$, $\pi = id$. This gives $\phi(z) = z^d$, and the action of $G_K$ on $T$ is determined by Kummer theory (or Artin-Schrier theory when $\phi$ is not separable). Let $K$ be a global field of characteristic not dividing $d$ and for given $n \geq 1$, let $\zeta_{d^n}$ be a fixed primitive $d^n$th root of unity and $\gamma_{d^n}$ a fixed root of $x^{d^n} - \alpha$. 
%$\mu_{d_n}^*$ the multiplicative group of all primitive $d^n$th roots of unity, and $\mu_{d^n}$ the multiplicative group of all $d^n$th roots of unity. 
Then there is a natural map
\begin{equation} \label{kummermap}
\omega: G_K \longrightarrow \invlim_{n \to \infty} \left( \begin{array}{cc} (\Z/d^n\Z)^* & \Z/d^n\Z \\  0 & 1 \end{array} \right)  
\end{equation}
given by 
$\sigma \mapsto \left( \begin{array}{cc} a & b \\  0 & 1 \end{array} \right)$ where $\sigma(\zeta_{d^n}) = \zeta_{d^n}^a$ and $\sigma(\gamma_{d^n})/\gamma_{d^n} = \zeta_{d^n}^b$ for each $n \geq 1$. Such an automorphism sends an arbitrary element $\zeta_{d^n}^i \gamma_{d^n}$ of $\phi^{-n}(\alpha)$ to $\zeta_{d^n}^{ai + b}\gamma_{d^n}$, and thus evidently the full group on the right-hand side of \eqref{kummermap} acts transitively on each $\phi^{-n}(\alpha)$, and thus on $\delta T$. 
When $\alpha \in K$ is not a root of unity, the image of $\omega$ in \eqref{kummermap} has finite index in the group on the right-hand side (see for example \cite[p. 302 and proof of Theorem 1]{ingram}), and hence $(z^d, \alpha)$ is eventually stable.  

\smallskip

\noindent \textbf{Case 2}: $A = \mathbb{G}_m$, $\delta(z) = z^d$ for $d \geq 2$, $\pi = z + z^{-1}$. In this case $\phi(z) = T_d(z)$, the monic degree-$d$ Chebyshev polynomial. If $\beta + \beta^{-1} = \alpha, \gamma_{d^n}$ is a root of $x^{d^n} - \beta$, and $\zeta_{d^n}$ is a primitive $d^n$th root of unity, then $$T_d^{-n}(\alpha) = \{\zeta_{d^n}^i \gamma_{d^n} + (\zeta_{d^n}^i \gamma_{d^n})^{-1} : i = 1, \ldots, d^n\}.$$ So if $G_K$ acts on the $\zeta_{d^n}^i \gamma_{d^n}$ with $j$ orbits, then it acts on $T_d^{-n}(\alpha)$ with at most $j$ orbits. Hence if $K$ is a global field of characteristic not dividing $d$, then by the Kummer-theoretic argument in Case 1, we have that $(T_d(z), \alpha)$ is eventually stable, provided that $\alpha$ is not the image of a root of unity under $\pi$. 

\smallskip

\noindent \textbf{Case 3}: $A = E$ is an elliptic curve defined over $K$ and $\delta$ is the multiplication-by-$\ell$ map for some prime $\ell$. In this case $\phi$ is a Latt\`es map, and $\pi$ can be taken to be defined over $\Q$ (\cite[Chapter 6]{jhsdynam}). %In this case, we may take $\pi$ to be either $(x, y) \mapsto x$, $(x,y) \mapsto x^2$, or $(x,y) \mapsto y$ (see \cite[Section ??]{jhsdynam}). 
Let $E[\ell^n] \subset E(\overline{K})$ be the $\ell^n$-torsion points of $E$, and suppose that $K$ is a global field of characteristic different from $\ell$, so that $E[\ell^n] \cong (\Z/\ell^n \Z)^2$. Let $\beta \in E(\overline{K})$ satisfy $\pi(\beta) = \alpha$, let $\ell^{-n}(\beta) = \{\gamma \in E(\overline{K}) : [\ell^n](\gamma) = \beta\}$, and fix $\gamma_n \in \ell^{-n}(\beta)$. Then there is a natural map
\begin{equation} \label{ecmap}
\omega: G_K \rightarrow \invlim_{n \to \infty} \left( \begin{array}{cc} \Aut(E[\ell^n]) &  E[\ell^n] \\  0 & 1 \end{array} \right)  \cong\left( \begin{array}{cc} \GL_2(\Z_\ell) &  \Z_\ell^2  \\  0 & 1 \end{array} \right), 
\end{equation}
where the first map is given by 
$\sigma \mapsto \left( \begin{array}{cc} a & b \\  0 & 1 \end{array} \right)$ where $a = \sigma|_{E[\ell^n]} $ and $b = \sigma(\gamma_n) - \gamma_n$ for each $n \geq 1$ (see \cite[Proposition 3.1]{itend}).
%$\sigma \mapsto \left( \begin{array}{cc} \sigma|_{E[\ell^n]}  & \sigma(\beta_n) - \beta_n \\  0 & 1 \end{array} \right)$. 
Such an automorphism sends an arbitrary element $u_n + \gamma_n \in \ell^{-n}(\beta)$, where $u_n \in E[\ell^n]$, to $(a(u_n) + b) + \gamma_n$, and thus evidently the full group on the right-hand side of \eqref{ecmap} acts transitively on each $\ell^{-n}(\beta)$. Suppose that $\beta$ is non-torsion, and that $E$ does not have complex multiplication. Then it follows from a well-known result of Serre \cite{serre1} and a result due to Bertrand \cite[Theorem 2, p. 40]{bertrand} (see \cite[Corollary 2.9]{pink} for a generalization) that the image of $\omega$ in \eqref{ecmap} has finite index in the right-hand side. Hence by Proposition \ref{elem}, $G_K$ acts on $\ell^{-n}(\beta)$ with a bounded number of orbits as $n$ grows. As in Case 2, we have that
$$
\phi^{-n}(\alpha) = \{\pi(\gamma) : \gamma \in \ell^{-n}(\beta)\}.
$$
Thus when $\alpha$ is not the image under $\pi$ of a torsion point, we have that $G_K$ acts on $\phi^{-n}(\alpha)$ with a bounded number of orbits as $n$ grows, so $(\phi, \alpha)$ is eventually stable.

\section{Proofs of Main Theorems} \label{proof}

Throughout this section, we use ``discrete valuation" to mean a discrete non-archimedean valuation on a field $K$, normalized so that $v(K \setminus \{0\}) = \Z$. We extend $v$ to a map on $\PP^1(K)$ by taking $v(0) = \infty$ and $v(\infty) = -\infty$. We denote by $R$ the ring of integers $\{x \in K : v(x) \geq 0\}$, and by $\p$ the unique maximal ideal $\{x \in K : v(x) > 0\}$ of $R$. We let $k$ be the residue field $R / \p$ and we denote by $\tilde{f}$ the polynomial obtained from $f \in R[z]$ by reducing each coefficient modulo $\p$. We begin with an easy generalization of Eisenstein's criterion.

\begin{lemma} \label{geneisenstein}
Let $v$ be a discrete valuation on a field $K$, let $f(z) = a_dz^d + \cdots + a_0 \in R[z]$ for $d \geq 1$, and suppose that $a_0 \neq 0$, $v(a_d) = 0$ and $v(a_i) > 0$ for all $i = 0, \ldots, d-1$. Then $f(z)$ has at most $v(a_0)$ irreducible factors over $K$. 
\end{lemma}

\begin{proof}
Write $f(z) = a_d f_0(z)$ with $f_0(z) \in R[z]$ monic, and let $v(f_0(0)) = m$, which is identical to $v(a_0)$. Suppose that $f_0(z) = g_1(z) \cdots g_{m+1}(z)$ is a factorization of $f_0$ into monics in $K[x]$ with $\deg g_i = e_i \geq 1$ and $\sum e_i = d$. By Gauss' Lemma, we may assume $g_i \in R[z]$ for all $i$. Then $\tilde{g_i}$ is again monic of degree $e_i$, and we have $z^d = \tilde{g_1}(z) \cdots \tilde{g_{m+1}}(z)$ in $k[z]$. Because $k[z]$ is a UFD, we must have $\tilde{g_i}(z) = z^{e_i}$ for all $i$, and hence $v(g_i(0)) > 0$ for all $i$. This implies $v(f_0(0)) = \sum v(g_i(0)) > m$, a contradiction. 
\end{proof}

The following lemma generalizes Lemma 2.2 of \cite{zdc}.

\begin{lemma} \label{genrigid}
Let $v$ be a discrete valuation on $K$ and let $\phi, \psi \in K(z)$ each have degree at least one. Suppose that $\phi, \psi$ have good reduction at $v$, $\phi(0) \neq 0$, $\psi(0) \neq 0$, $v(\phi'(0)) > 0$, and $v(\phi(0)) = v(\psi(0)) = r > 0$. Then $v(\phi(\psi(0))) = r$.
\end{lemma}

\begin{remark}
The assumption that $\phi$ and $\psi$ have good reduction can be replaced by assuming that the constant term in the denominator of each function is in $R^*$. This assumption is necessary, as illustrated by the example of $K=  \Q$, $\phi(z) = \psi(z) = (z^2 + 8)/4$, and $v$ the $2$-adic valuation.
\end{remark}

\begin{proof} 
Let $d = \deg \phi$ and $d' = \deg \psi$, and write
\begin{equation*}
\phi(z) = \frac{a_{d}z^{d} + \ldots + a_{0}}{b_{d}z^{d} + \ldots + b_{0}}, \qquad \psi(z) = \frac{a'_{d}z^{d'} + \ldots + a'_{0}}{b'_{d}z^{d'} + \ldots + b'_{0}} = \frac{f(z)}{g(z)},
\end{equation*}
with all coefficients in $R$, at least one of the $a_i, b_j$ in $R^*$ and at least one of the $a'_i, b'_j$ in $R^*$. Because $\phi(0) \neq 0$, we have $a_0 \neq 0$, and because $v(\phi(0)) > 0$ we have $\phi(0) \neq \infty$, and so $b_0 \neq 0$. By assumption $0 < r = v(\phi(0)) = v(a_0) - v(b_0)$, and since $b_0 \in R$ we must have $v(a_0) > 0$. Then because $\phi$ has good reduction at $v$ we must also have $v(b_0) = 0$, whence $v(a_0) = r$. Similarly, $v(a'_0) = r$ and $v(b'_0) = 0$.  Note also that $\phi'(0) = (b_0a_1 - a_0b_1)/b_0^2$, and because $v(b_0) = 0$, $v(a_0b_1) > 0$, and $v(\phi'(0)) > 0$, we must have $v(a_1) > 0$. Now 
$$
\phi(\psi(z)) =  \frac{a_{d}\psi(z)^{d} + \ldots + a_1 \psi(z) + a_{0}}{b_{d}\psi(z)^{d} + \ldots + b_1\psi(z) + b_{0}} = 
 \frac{a_{d}f(z)^{d} + \ldots + a_1 f(z)g(z)^{d-1} + a_{0}g(z)^d}{b_{d}f(z)^{d} + \ldots + b_1f(z)g(z)^{d-1} + b_{0}g(z)^d},
$$
and hence
\begin{equation} \label{constterm}
\phi(\psi(0)) =  \frac{a_{d}(a'_0)^{d} + \ldots + a_1 a'_0(b'_0)^{d-1} + a_{0}(b'_0)^d}{b_{d}(a'_0)^{d} + \ldots + b_1a'_0 (b'_0)^{d-1} + b_{0}(b'_0)^d}.
\end{equation}
Because $v(a'_0) = r > 0$, all terms in the right-hand side of \eqref{constterm} have valuation at least $2r$ except $a_1 a'_0(b'_0)^{d-1}$ and $a_{0}(b'_0)^d$ in the numerator and 
$b_1a'_0 (b'_0)^{d-1}$ and  $b_{0}(b'_0)^d$ in the denominator. But $v(a_1) > 0$, and so $v(a_1 a'_0(b'_0)^{d-1}) > r$, implying that the numerator has valuation $r$. On the other hand, $v(b_{0}(b'_0)^d) = 0$ and $v(b_1a'_0 (b'_0)^{d-1}) > 0$, so the denominator has valuation $0$. 
\end{proof}

\begin{lemma} \label{genrigid2}
Let $v$ be a discrete valuation on $K$, let $\phi \in K(z)$ have degree $d \geq 1$, and suppose that $\phi$ has good reduction at $v$, $\phi(0) \neq 0,$ $v(\phi(0)) > 0$, and $v(\phi'(0)) > 0$. Then $v(\phi^n(0)) = v(\phi(0))$ for all $n \geq 1$. 
\end{lemma}

\begin{proof}
A straightforward inductive argument, taking $\psi = \phi^{n-1}$ in Lemma \ref{genrigid}. %[Something about the $\neq 0$ hypothesis.]
%The statement is trivial for $n = 1$. Assume inductively that $v(\phi^{n-1}(0)) = v(\phi(0))$ (implying $\phi^{n-1}(0) \neq 0$) and $v(\phi^{n-1}'(0)) > 0$. Taking $\psi = \phi^{n-1}$ in Lemma \ref{genrigid} gives the result. 
\end{proof}

To give some intuition for why Lemma \ref{genrigid2} is true, note that the hypotheses that $\phi$ have good reduction at $v$, $\phi(0) \neq 0$, and $v(\phi(0)) > 0$ ensure that $\p$ contains a fixed point of $\phi$ whose multiplier vanishes modulo $\p$. If we define a $\p$-adic absolute value on $K \setminus \{0\}$ by setting $|x| = p^{-v(x)}$ in the usual way, then such a fixed point is $\p$-adically attracting, and $0$ lies in its $\p$-adic basin of attraction. Thus $\p$-adically the orbit of $0$ converges monotonically to this fixed point, and by the strong triangle inequality every element of this orbit must have constant absolute value.

For $\phi \in K(z)$, we say that $\phi(z) = f(z)/g(z)$ is \textit{normalized} if $f, g \in R[z]$ are coprime and at least one coefficient of $f$ or $g$ is in $R^*$. Because the next result deals with eventual stability, we state it only for $d \geq 2$. 

\begin{theorem} \label{evstab1}
Let $v$ be a discrete valuation on $K$, let $\phi \in K(z)$ have degree $d \geq 2$, and let $\phi(z) = f(z)/g(z)$ be normalized. Suppose that $\phi$ has good reduction at $v$, $\phi(0) \neq 0$, and $\tilde{f} = Cz^d$ for $C \in k^*$. For each $n \geq 1$, let $\phi^n(z) = f_n(z)/g_n(z)$ be normalized. Then for all $n \geq 1$, $f_n(z)$ has at most $v(\phi(0))$ irreducible factors over $K$. In particular, $\phi$ is eventually stable over $K$. 
\end{theorem}

\begin{proof}  
Because $\phi$ has good reduction at $v$, we have $\widetilde{\phi^n}(z) = \tilde{\phi}^n(z)$ \cite[Theorem 2.18]{jhsdynam}. Because $\phi^n(z) = f_n(z)/g_n(z)$ is normalized, we have $\tilde{\phi}^n(z) = \tilde{f_n}(z)/\tilde{g_n}(z)$. However, the assumption that $\tilde{\phi}(z) = Cz^d/\tilde{g}(z)$ implies that $\tilde{\phi}^n(z) = C_nz^{d^n}/\tilde{g_n}(z)$ for some $C_n \in k^*$, and thus $\tilde{f_n} = C_nz^{d^n}$. But $\deg f_n \leq \deg \phi^n = d^n$, and so $\deg f_n = d^n$. From Lemma \ref{geneisenstein} we now have that $f_n$ has at most $v(f_n(0))$ factors, provided that $f_n(0) \neq 0$. But $\phi^n$ has good reduction at $v$, and hence $v(g_n(0)) = 0$. It follows that $v(f_n(0)) = v(\phi^n(0)),$ which by Lemma \ref{genrigid2} must equal $v(\phi(0))$. In particular $\phi^n(0) \neq 0$, and it follows that $f_n$ has at most $v(\phi^n(0)) = v(\phi(0))$ irreducible factors over $K$. 
\end{proof}

%In the following, we take $v(\infty) = -\infty$.

\begin{theorem} \label{evstab2}
Let $v$ be a discrete valuation on $K$, let $\phi \in K(z)$ have degree $d \geq 2$, and let $\alpha \in \mathbb{P}^1(K)$. Suppose that $\phi$ has good reduction at $v$, $\phi(\alpha) \neq \alpha$, and $\tilde{\phi}^{-1}(\tilde{\alpha}) = \{\tilde{\alpha}\}$ as a map of $\mathbb{P}^1(\overline{k})$. For each $n \geq 1$, let $\phi^n = f_n/g_n$ be normalized. Then $f_n(z) - \alpha g_n(z)$ 
(or $g_n(z)$ if $\alpha = \infty$) has at most $v(\phi(\alpha) - \alpha)$ (or $v(\phi(\alpha)^{-1} - \alpha^{-1})$ if $v(\alpha) < 0$) irreducible factors over $K$. In particular, $(\phi, \alpha)$ is eventually stable over $K$. 
\end{theorem}

\begin{proof}
Let $\phi = f/g$ be normalized. First suppose that $\alpha \in R$, and let $\phi_0(z) = \phi(z + \alpha) - \alpha$. Then the map $\mu(z) = z + \alpha$ has good reduction at $v$, and because $\phi_0 = \mu^{-1} \circ \phi \circ \mu$ it follows from \cite[Theorem 2.18]{jhsdynam} that $\phi_0$ has good reduction at $v$. 
The conditions that $\phi$ have good reduction at $v$ and $\phi^{-1}(\tilde{\alpha}) = \{\tilde{\alpha}\}$ as a map of $\mathbb{P}^1(\overline{k})$ are equivalent to $\tilde{\phi}(z) - \tilde{\alpha}= C(z - \tilde{\alpha})^d/\tilde{g}(z)$ for $C \in k^*$, implying that $\widetilde{\phi_0}(z) = Cz^d/\tilde{g}(z + \tilde{\alpha})$. Finally, $\phi(\alpha) \neq \alpha$ implies $\phi_0(0) \neq 0$. Letting $\phi_0^n = u_n/t_n$ be normalized, Theorem \ref{evstab1} gives that $u_n(z)$ has at most $v(\phi_0(0))$ irreducible factors over $K$. Now if $\phi^n = f_n/g_n$ is normalized, then $\phi^n(z + \alpha) - \alpha = u_n(z)/t_n(z)$ gives $u_n(z) = f_n(z + \alpha) - \alpha g_n(z + \alpha)$, and thus the number of irreducible factors of $u_n(z)$ over $K$ is the same as the number of irreducible factors of $f_n(z) - \alpha g_n(z)$ over $K$. Because $v(\phi_0(0)) = v(\phi(\alpha) - \alpha)$, the theorem is proved for $\alpha \in R$. 

Assume now that $\alpha \in (K \setminus R) \cup \{\infty\}$ (i.e. $\tilde{\alpha} = \infty$), let $\psi(z) = 1/\phi(1/z)$, and take $1/\alpha = 0$ if $\alpha = \infty$. Then $1/\alpha \in R$, $\psi$ has good reduction at $v$, $\psi(1/\alpha) \neq 1/\alpha$, and $\tilde{\psi}^{-1}(1/\tilde{\alpha}) = \{1/\tilde{\alpha}\}$ as a map of $\mathbb{P}^1(\overline{k})$. Letting $\psi^n = v_n/w_n$ be normalized, the previous paragraph gives that $v_n(z) - (1/\alpha)w_n(z)$ has at most $v(\psi(1/\alpha) - 1/\alpha)$ irreducible factors over $K$. But $v_n(z) - (1/\alpha)w_n(z) = g_n^*(z) - (1/\alpha)f_n^*(z)$, where $f_n^*, g_n^*$ denote the reciprocal polynomials of $f_n, g_n$, and hence $v_n(z) - (1/\alpha)w_n(z)$ has the same number of irreducible factors over $K$ as $g_n(z) - (1/\alpha)f_n(z)$. If $\alpha = \infty$, then this bounds the number of irreducible factors of $g_n(z)$, while if $\alpha \neq \infty$ it bounds the number of irreducible factors of $(-\alpha)(g_n(z) - (1/\alpha)f_n(z)) = f_n(z) - \alpha g_n(z)$. The bound is given by $v(\psi(1/\alpha) - 1/\alpha)$, which is $v(1/\phi(\alpha) - 1/\alpha)$. 
\end{proof}

Theorem \ref{evstab2} allows us to generalize \cite[Corollary 3]{ingram} on the eventual stability of certain polynomials. Moreover, our proof is purely algebraic, in contrast to the analytic arguments of \cite{ingram}, where the author constructs a Galois-equivariant $p$-adic version of the B\"ottcher coordinate from complex dynamics.

\begin{corollary} \label{polycor}
Let $K$ be a field, let $v$ be a discrete valuation on $K$, and let $\phi \in K[z]$ have degree $d \geq 2$. If $\phi$ has good reduction at $v$ and $\alpha \in K$ has $v(\alpha) < 0$, then $\phi^n(z) - \alpha$ has at most 
$-v(\alpha)$ irreducible factors over $K$. In particular, $(\phi, \alpha)$ is eventually stable. 
\end{corollary}

\begin{proof}
Note that the hypotheses prohibit $\phi(\alpha) = \alpha$. Moreover, $\tilde{\phi}$ is a polynomial, and hence satisfies $\tilde{\phi}^{-1}(\infty) = \{\infty\}$ as a map of $\mathbb{P}^1(\overline{k})$. The result now follows from Theorem \ref{evstab2} and the observation that for a polynomial $a_dz^d + \cdots + a_0$ with $v(a_d) = 0$ and $v(a_i) \geq 0$ for $i = 1, \ldots, d-1$, we have $v(1/\phi(\alpha) - 1/\alpha) = v(\phi(\alpha) - \alpha) - v(\alpha \phi(\alpha))$, and $v(\alpha) > 0$ and the strong triangle inequality force the last expression to be $dv(\alpha) - (d+1)v(\alpha)$.
\end{proof}

\begin{definition}
Let $v$ be a discrete valuation on a field $K$, and assume that the reside field $k$ is finite of characteristic $p$. We say that $\phi \in K(z)$ is \textbf{bijective on residue extensions} for v if $\tilde{\phi}$ acts as a bijection on $\mathbb{P}^1(E)$ for every finite extension $E$ of $k$.
\end{definition}

The property of being bijective on residue extensions is a strong form of the notion of exceptional rational function, i.e., a rational function that acts bijectively on infinitely many residue extensions (see \cite{gtzexceptional} for a further generalization of this definition to certain maps of varieties). We now give a characterization of maps that are bijective on residue extensions. While this characterization is well-known to experts, we include a proof for completeness. For polynomials this is a result of Carlitz \cite[Theorem 3]{carlitzperm}, and the proof we give closely resembles his argument. 

\begin{proposition} \label{bijective reduction characterization}
Let $v$ be a discrete valuation on a field $K$, suppose that the residue field $k$ is finite of characteristic $p$, and let $\phi \in K(z)$ have degree $d \geq 1$. The following are equivalent: 
\begin{enumerate}
\item[(A)] $\phi$ is bijective on residue extensions for $v$.
\item[(B)] for each $\beta \in \mathbb{P}^1(k)$, there is a unique $\gamma \in \mathbb{P}^1(\overline{k})$ with $\tilde{\phi}(\gamma) = \beta$ (which must in fact satisfy $\gamma \in \PP^1(k)).$  
\item[(C)] $\tilde{\phi}$ is non-constant and is of the form $(c_1z^{p^j} + c_2)/(c_3z^{p^j} + c_4)$ for $j \geq 0$ and $c_1, c_2, c_3, c_4 \in k$.
\end{enumerate}
\end{proposition}

\begin{proof}
To show (A) implies (B), let $E$ be the finite extension of $k$ given by adjoining all $\gamma \in \overline{k}$ with $\tilde{\phi}(\gamma) = \beta$ for some $\beta \in \mathbb{P}^1(k)$. By (B) we have that $\tilde{\phi}$ acts on $E$ as a bijection, from which it follows that if $\tilde{\phi}(z) = \beta$ has a solution in $\overline{k}$, it has a unique such solution. Moreover, if $\beta \in \mathbb{P}^1(k)$ is such that $\tilde{\phi}(z) = \beta$ has no solution in $\overline{k}$, then $\tilde{\phi}^{-1}(\beta) = \{\infty\}$ as a map of $\PP^1(\overline{k})$. This shows that as a map of $\overline{k}$, we have $\#\tilde{\phi}^{-1}(\PP^1(k)) =  \#(\PP^1(k))$. But $\tilde{\phi}(\PP^1(k)) = \PP^1(k)$, and so we conclude that $\tilde{\phi}^{-1}(\PP^1(k)) =  \PP^1(k)$ and hence $\gamma \in \PP^1(k)$.

To show that (B) implies (C), first note that (B) precludes $\tilde{\phi}$ from being constant. Let $e = \deg \tilde{\phi} \geq 1$, write $\tilde{\phi}(z) = f(z)/g(z)$ with $f, g$ coprime, and consider the set $S = \{g(z)\} \cup \{f(z) - \beta g(z) : \beta \in k\}$. It follows from (B) that precisely one element of $S$ is constant, and all others have the form $c(z-\gamma)^e$ for $\gamma \in k$ and $c \neq 0$.  Thus
\begin{equation} \label{forms}
\text{$\tilde{\phi}$ has one of the forms} \quad \frac{c_1(z - \gamma_1)^e}{c_3(z - \gamma_2)^e}, \quad \frac{c_1}{c_3(z - \gamma_2)^e}, \quad \text{or} \quad \frac{c_1(z - \gamma_1)^e}{c_3},
\end{equation}
with $c_1c_3 \neq 0$ and $\gamma_1, \gamma_2 \in k, \gamma_1 \neq \gamma_2$. In light of the equations $g(z) + (f(z) - g(z)) = f(z), f(z) - (f(z) - g(z)) = g(z)$, and $f(z) - (\beta g(z)) = f(z) - \beta g(z)$, we then have 
\begin{equation*}
C_1(z + \gamma_1)^e + C_2(z + \gamma_2)^e = 1,
\end{equation*}
where $C_1, C_2 \in k$, $C_1C_2 \neq 0$, and we substitute $-\gamma_i$ for $\gamma_i$ in order to ease notation, which still preserves $\gamma_1 \neq \gamma_2$. Matching $z^e$-coefficients gives $C_1 + C_2 = 0$. Matching lower-degree coefficients then gives
\begin{equation} \label{binom}
C_1 {e\choose{e-i}} (\gamma_1^i - \gamma_2^i) = 0
\end{equation}
for $i = 0, \ldots, e-1$.  Now let $j \geq 0$ be such that $p^j \mid e$. If $e > p^j$, then set $i = p^j$ in \eqref{binom} and note that $z \mapsto z^{p^j}$ induces a bijection on $\overline{k}$.  Thus $\gamma_1 \neq \gamma_2$ implies $\gamma_1^{p^j} \neq \gamma_2^{p^j}$, and so  
${e \choose{e - p^j}} = 0$ in $k$, giving
\begin{equation} \label{divis}
p \mid \frac{e(e-1) \cdots (e - p^j + 1)}{p^j(p^j-1) \cdots 1}.
\end{equation}
But the assumption that $p^j \mid e$ implies that $v_p(e-i) = v_p(p^j - i)$ for all $i = 1, \ldots, p^j - 1$, where $v_p$ denotes the $p$-adic valuation. It then follows from \eqref{divis} that $p \mid (e/p^j)$, and thus $p^{j+1} \mid e$. 
Now assume that $e = p^j \cdot m$, where $m \geq 1$ is maximal subject to $p \nmid m$. If $m > 1$, then $e > p^j$, and so $p^{j+1} \mid d$, a contradiction. 
Hence $e = p^j$.  It now follows immediately from \eqref{forms} that $\tilde{\phi}$ has the form given in (C). 

To show (C) implies (A), let $E$ be a finite extension of $k$, and $\beta \in E$. Because $\tilde{\phi}$ is non-constant, the equation $\tilde{\phi}(z) = \beta$ has no solutions in $E$ if and only if 
$c_3 \neq 0$ and $\beta = c_1/c_3$, in which case $\tilde{\phi}(\infty) = \beta$. Otherwise $\tilde{\phi}(z) = \beta$ is equivalent to $z^{p^j} = \delta$, where $\delta = -(c_2 - \beta c_4)/(c_1 - \beta c_3) \in k$. But $E$ is a finite field of characteristic $p$, and so $z \mapsto z^{p^j}$ gives a bijection on $E$, whence $z^{p^j} = \delta$ has a (unique) solution in $E$. This shows that $E$ is contained in the image of the map $\tilde{\phi} : \PP^1(E) \to \PP^1(E)$. But $\tilde{\phi}(\infty) = \infty$ if $c_3 = 0$, and otherwise $\tilde{\phi}(\gamma) = \infty$, where $\gamma$ is the unique solution in $E$ to $z^{p^j} = -c_3/c_4$. Hence $\tilde{\phi}$ is a surjection of the finite set $\PP^1(E)$ to itself, and thus a bijection. 
\end{proof}

\begin{corollary} \label{fullmain}
Let $v$ be a discrete valuation on a field $K$, suppose that $k$ is finite of characteristic $p$, let $\phi \in K(z)$ have degree $d \geq 2$, and let $\alpha \in \mathbb{P}^1(K)$ be non-periodic for $\phi$. Suppose that $\phi$ has good reduction at $v$ and is bijective on residue extensions for $v$.  If $\phi^n = f_n/g_n$ is normalized, then $f_n(z) - \alpha g_n(z)$ (or $g_n(z)$ if $\alpha = \infty$) has at most $v(\phi^i(\alpha) - \alpha)$ (or $v(\phi(\alpha)^{-1} - \alpha^{-1})$ if $v(\alpha) < 0$) irreducible factors over $K$, where
$$
i = \min \{n \geq 1 : \tilde{\phi}^n(\tilde{\alpha}) = \tilde{\alpha} \} \leq \#(\PP^1(k)) .
$$
In particular, $(\phi, \alpha)$ is eventually stable, and hence if $K$ is a number field or a function field, Conjecture \ref{mainconj} holds for $\phi$. 
\end{corollary}

\begin{proof}
Because $\tilde{\phi}$ acts bijectively on $\PP^1(k)$, there exists $i \geq 1$ with $\tilde{\phi}^i(\tilde{\alpha}) = \tilde{\alpha}$. Assume that $i$ is the minimal such integer, and note that $i \leq \#(\PP^1(k))$.  By condition (B) of Proposition \ref{bijective reduction characterization}, we have $(\tilde{\phi}^i)^{-1}(\tilde{\alpha}) = \{\tilde{\alpha}\}$. We also have $\phi^i(\alpha) \neq \alpha$, because $\alpha$ is assumed to be non-periodic for $\phi$. The corollary then follows from Theorem \ref{evstab2} and $(c)$ of Proposition \ref{charprop1}.
\end{proof}

\section*{Acknowledgements}

\noindent The authors would like to thank Michael Zieve for pointing out useful references. 
The second author's research was partially supported by the G\"{o}ran Gustafson Foundation.

\bibliographystyle{plain}

\begin{thebibliography}{10}

\bibitem{shparostafe}
Omran Ahmadi, Florian Luca, Alina Ostafe, and Igor~E. Shparlinski.
\newblock On stable quadratic polynomials.
\newblock {\em Glasg. Math. J.}, 54(2):359--369, 2012.

\bibitem{Beardon}
Alan~F. Beardon.
\newblock {\em Iteration of rational functions}, volume 132 of {\em Graduate
  Texts in Mathematics}.
\newblock Springer-Verlag, New York, 1991.
\newblock Complex analytic dynamical systems.

\bibitem{bertrand}
D.~Bertrand.
\newblock Galois representations and transcendental numbers.
\newblock In {\em New advances in transcendence theory (Durham, 1986)}, pages
  37--55. Cambridge Univ. Press, Cambridge, 1988.

\bibitem{settled}
Nigel Boston and Rafe Jones.
\newblock Settled polynomials over finite fields.
\newblock {\em Proc. Amer. Math. Soc.}, 140(6):1849--1863, 2012.

\bibitem{bridytucker}
Andrew Bridy and Thomas~J. Tucker.
\newblock {Iterated {G}alois groups of general cubic polynomials}.
\newblock {In preparation}.

\bibitem{carlitzperm}
L.~Carlitz.
\newblock Permutations in finite fields.
\newblock {\em Acta Sci. Math. (Szeged)}, 24:196--203, 1963.

\bibitem{danielson}
Lynda Danielson and Burton Fein.
\newblock On the irreducibility of the iterates of {$x^n-b$}.
\newblock {\em Proc. Amer. Math. Soc.}, 130(6):1589--1596 (electronic), 2002.

\bibitem{preimcurves}
Xander Faber, Benjamin Hutz, Patrick Ingram, Rafe Jones, Michelle Manes,
  Thomas~J. Tucker, and Michael~E. Zieve.
\newblock Uniform bounds on pre-images under quadratic dynamical systems.
\newblock {\em Math. Res. Lett.}, 16(1):87--101, 2009.

\bibitem{ostafe2}
Domingo G{\'o}mez-P{\'e}rez, Alejandro~P. Nicol{\'a}s, Alina Ostafe, and Daniel
  Sadornil.
\newblock Stable polynomials over finite fields.
\newblock {\em Rev. Mat. Iberoam.}, 30(2):523--535, 2014.

\bibitem{ostafe}
Domingo G{\'o}mez-P{\'e}rez, Alina Ostafe, and Igor~E. Shparlinski.
\newblock On irreducible divisors of iterated polynomials.
\newblock {\em Rev. Mat. Iberoam.}, 30(4):1123--1134, 2014.

\bibitem{abc}
C.~Gratton, K.~Nguyen, and T.~J. Tucker.
\newblock {$ABC$} implies primitive prime divisors in arithmetic dynamics.
\newblock {\em Bull. Lond. Math. Soc.}, 45(6):1194--1208, 2013.

\bibitem{gtzexceptional}
Robert~M. Guralnick, Thomas~J. Tucker, and Michael~E. Zieve.
\newblock Exceptional covers and bijections on rational points.
\newblock {\em Int. Math. Res. Not. IMRN}, (1):Art. ID rnm004, 20, 2007.

\bibitem{zdc}
Spencer Hamblen, Rafe Jones, and Kalyani Madhu.
\newblock The density of primes in orbits of {$z^d+c$}.
\newblock {\em Int. Math. Res. Not. IMRN}, (7):1924--1958, 2015.

\bibitem{ingram}
Patrick Ingram.
\newblock Arboreal {G}alois representations and uniformization of polynomial
  dynamics.
\newblock {\em Bull. Lond. Math. Soc.}, 45(2):301--308, 2013.

\bibitem{jacobs}
Konrad Jacobs.
\newblock {\em Measure and integral}.
\newblock Academic Press [Harcourt Brace Jovanovich, Publishers], New
  York-London, 1978.
\newblock Probability and Mathematical Statistics, With an appendix by Jaroslav
  Kurzweil.

\bibitem{quaddiv}
Rafe Jones.
\newblock The density of prime divisors in the arithmetic dynamics of quadratic
  polynomials.
\newblock {\em J. Lond. Math. Soc. (2)}, 78(2):523--544, 2008.

\bibitem{itconst}
Rafe {Jones}.
\newblock {An iterative construction of irreducible polynomials reducible
  modulo every prime}.
\newblock {\em J. Algebra}, 369:114--128, 2012.

\bibitem{survey}
Rafe Jones.
\newblock Galois representations from pre-image trees: an arboreal survey.
\newblock In {\em Actes de la {C}onf\'erence ``{T}h\'eorie des {N}ombres et
  {A}pplications''}, Publ. Math. Besan\c con Alg\`ebre Th\'eorie Nr., pages
  107--136. Presses Univ. Franche-Comt\'e, Besan\c con, 2013.

\bibitem{galrat}
Rafe Jones and Michelle Manes.
\newblock Galois theory of quadratic rational functions.
\newblock {\em Comment. Math. Helv.}, 89(1):173--213, 2014.

\bibitem{itend}
Rafe Jones and Jeremy Rouse.
\newblock Galois theory of iterated endomorphisms.
\newblock {\em Proc. Lond. Math. Soc. (3)}, 100(3):763--794, 2010.
\newblock Appendix A by Jeffrey D. Achter.

\bibitem{langdioph}
Serge Lang.
\newblock {\em Fundamentals of {D}iophantine geometry}.
\newblock Springer-Verlag, New York, 1983.

\bibitem{narkiewicz}
W{\l}adys{\l}aw Narkiewicz.
\newblock {\em Elementary and analytic theory of algebraic numbers}.
\newblock Springer Monographs in Mathematics. Springer-Verlag, Berlin, third
  edition, 2004.

\bibitem{odonigalit}
R.~W.~K. Odoni.
\newblock The {G}alois theory of iterates and composites of polynomials.
\newblock {\em Proc. London Math. Soc. (3)}, 51(3):385--414, 1985.

\bibitem{pellet}
A.-E. {Pellet}.
\newblock {Sur la d\'ecomposition d'une fonction enti\`ere en facteurs
  irr\'eductibles suivant un module premier $p$}.
\newblock {\em C. R. Acad. Sci. Paris}, 86:1071--1072, 1878.

\bibitem{pink}
Richard Pink.
\newblock On the order of the reduction of a point on an abelian variety.
\newblock {\em Math. Ann.}, 330(2):275--291, 2004.

\bibitem{serre1}
Jean-Pierre Serre.
\newblock Propri\'et\'es galoisiennes des points d'ordre fini des courbes
  elliptiques.
\newblock {\em Invent. Math.}, 15(4):259--331, 1972.

\bibitem{jhsintegers}
Joseph~H. Silverman.
\newblock Integer points, {D}iophantine approximation, and iteration of
  rational maps.
\newblock {\em Duke Math. J.}, 71(3):793--829, 1993.

\bibitem{jhsdynam}
Joseph~H. Silverman.
\newblock {\em The arithmetic of dynamical systems}, volume 241 of {\em
  Graduate Texts in Mathematics}.
\newblock Springer, New York, 2007.

\bibitem{sookdeo}
Vijay~A. Sookdeo.
\newblock Integer points in backward orbits.
\newblock {\em J. Number Theory}, 131(7):1229--1239, 2011.

\bibitem{stickelberger}
L.~Stickelberger.
\newblock \"{U}ber eine neue {E}igenschaft der {D}iskriminanten algebraischer
  {Z}ahlk\"orper.
\newblock {\em Verh. 1 Internat. Math. Kongresses, 1897, Leipzig}, pages
  182--193, 1898.

\end{thebibliography}

\end{document}